\documentclass[a4paper,11pt,reqno,noindent]{amsart}
\usepackage[centertags]{amsmath}
\usepackage{amsfonts,amssymb,amsthm} 
\usepackage{hyperref}

 \hypersetup{
     pdfmenubar=false,        % show Acrobat¡¯s menu?
     pdfnewwindow=true,      % links in new window
     colorlinks=false,       % false: boxed links; true: colored links
     linkcolor=blue,          % color of internal links
     citecolor=blue,        % color of links to bibliography
     filecolor=magenta,      % color of file links
     urlcolor=cyan           % color of external links
 }
\usepackage{graphicx} 
\usepackage[numbers,sort]{natbib}
\usepackage[english]{babel}
\usepackage{newlfont}
\usepackage{color}
\usepackage[body={15cm,21.5cm},centering]{geometry} 
\usepackage{fancyhdr}
\usepackage{esint}
\usepackage{enumerate}

\usepackage[active]{srcltx}

\theoremstyle{plain}
\newtheorem{theorem}{Theorem}[section]
\newtheorem{lemma}[theorem]{Lemma}

\newtheorem{definition}[theorem]{Definition}

\theoremstyle{definition}
\newtheorem{remark}[theorem]{Remark}

\newcommand{\sym}{\mathrm{sym}}

\newcommand{\supp}{\operatorname{supp}}

\newcommand{\e}{\varepsilon}

\newcommand{\degd}{\mathrm{deg}^\partial}
\newcommand{\R}{\mathbb{R}}
\newcommand{\Z}{\mathbb{Z}}
\newcommand{\N}{\mathbb{N}}

\renewcommand{\d}{\text{d}}

\newcommand{\A}{\mathcal{A}}

\newcommand{\id}{\mathrm{Id}}

\newcommand{\curl}{{\mathrm{curl}\,}}

\renewcommand{\H}{\mathcal{H}}

\renewcommand{\div}{\mathrm{div}\,}

\newcommand{\cof}{\mathrm{cof}\,}
\newcommand{\Ihp}{I_{h,p}}

\title{On a boundary value problem for conically deformed thin elastic sheets}

\author{Heiner Olbermann}

\begin{document}

\maketitle

\begin{abstract}
We consider a thin elastic sheet in the shape of a disk that is clamped at its
boundary such that the displacement and the deformation gradient coincide with a
conical deformation with no stretching there. These are the boundary conditions
of a so-called ``d-cone''. We define the free elastic energy
as a variation of the von K\'arm\'an energy, that penalizes bending energy in
$L^p$ with $p\in (2,\frac83)$ (instead of, as usual, $p=2$). We prove ansatz free
upper and lower bounds for the elastic energy that scale like $h^{p/(p-1)}$, where $h$ is the thickness of the sheet.
\end{abstract}

\section{Introduction}

Strong deformations of thin elastic sheets under the influence of some
external force have been a topic of considerable interest in the physics and
engineering community over the last decades. These ``post-buckling'' phenomena
are relevant on many length scales, e.g., for structural failure, for the design
of protective  structures, or in atomic force microscopy of virus capsids and bacteria. In the physics literature, one finds numerous contributions
that discuss the focusing of elastic energy in  ridges and conical vertices, see
\cite{CCMM,MR2023444,1997PhRvE..55.1577L}. The overview article by Witten
\cite{RevModPhys.79.643} contains a comprehensive review of the activities in
that area of physics. However, quoting the seminal work
\cite{Lobkovsky01121995}, the ``understanding of the strongly buckled state
remains primitive'', and this fact has not changed fundamentally since the
publication of that article more than 20 years ago.

\medskip

In the mathematical literature on thin elastic sheets, there have been two major
topics: On the one hand, there are the  derivations of  lower dimensional models starting from
three-dimensional finite elasticity \cite{ciarlet1997theory,MR1365259,MR1916989,MR2210909}. On the
other hand, there has been quite some effort to investigate the qualitative
properties of plate models by determining the scaling behavior of the free
elastic energy with respect to the small parameters in the model (such as the
thickness of the sheet). Such scaling laws have been derived, e.g., in
\cite{MR3179665,MR1921161,2015arXiv151207416B,kohn2013analysis}. Building
on the results from \cite{MR2023444}, it has been proved in
\cite{MR2358334}  that the free energy per unit thickness of
the so-called ``single fold'' scales like $h^{5/3}$, where $h$ is the thickness
of the sheet. This is also the conjectured scaling behavior for the confinement
problem, which consists in determining the minimum of elastic energy necessary
to fit a thin elastic sheet into a container whose size is smaller than the
diameter of the sheet. The energy focusing in conical vertices has been
investigated in \cite{MR3102597,MR3168627}, where the following has been proved:
Consider a thin elastic sheet in the shape of a disc, and fix it at the
boundary and at the center such
that it agrees with a (non-flat) conical configuration there. Then the elastic
energy scales like $h^2\log \frac1h$. On a technical level, the papers \cite{MR2358334,MR3102597,MR3168627}
consider an energy functional of the form
\begin{equation}
I_h(y)=\int_{\Omega} |Dy^TDy-\id_{2\times 2}|^2+h^2|D^2y|^2\d x\,,\label{eq:14}
\end{equation}
where $\Omega\subset\R^2$ is the undeformed sheet, $y:\Omega\to\R^3$ is the
deformation, and  $\id_{2\times 2}$ is the the 2 by 2 identity matrix.
The first term is the (non-convex) membrane energy, and the second is the
bending energy.
 If one manages to derive scaling laws for this two-dimensional model, then
as a consequence, it is often the case that analogous results for three-dimensional elasticity are  not difficult to
derive as a corollary by the results from \cite{MR1916989}, see for example
\cite{MR2358334,MR3102597}.  Of course,
the character of the variational problem  heavily depends on the chosen
boundary conditions.

\medskip

While the mentioned articles have contributed a lot to the mathematical
understanding of folds and vertices in thin sheets, % they do not address the
% question of why these patterns \emph{emerge} under compressive forces. More precisely,
they do not consider situations where the constraints prevent the sheet  from adopting  an
isometric immersion with respect to the reference metric as its configuration,
but do not prevent it from adopting a \emph{short}  map as its configuration. (We
recall that a map $y:\Omega\to\R^3$ is short  if
every path $\gamma\subset\Omega$ is mapped to a shorter path $y(\gamma)\subset\R^3$.)  Such a
situation is characteristic of post-buckling, and in particular, the confinement
problem. 

\medskip

The reason why short maps are problematic can be found in the famous Nash-Kuiper Theorem
\cite{MR0065993,MR0075640}: If one is given a short map $y_0\in C^1(\Omega;\R^3)$ and $\e>0$, then
there exists an isometric immersion $y\in C^1(\Omega;\R^3)$ with
$\|y-y_0\|_{C^{0}}<\e$. This is relevant in the present context, since the
difference between the induced metric and the flat reference metric is the leading order term in the energy
\eqref{eq:14}. Thus, if short maps are permissible, then there exists a vast
amount of configurations with vanishing or very small membrane energy. One needs a principle
that is capable to show that
all these maps are associated with a large amount of bending energy.
As has recently been shown in \cite{lewicka2015convex}, this problem is  not
only encountered when dealing with the geometrically fully nonlinear plate model \eqref{eq:14}. It is
also present in the  von K\'arm\'an model, which we are
going to treat here. In fact, the proof in \cite{lewicka2015convex} is based on
a suitable adaptation of the Nash-Kuiper argument to the von K\'arm\'an model.

\medskip

Possibly the simplest example of a variational problem where isometric immersions are prohibited by the boundary
conditions, but short maps are not, is given
by a modification of the ``conically constrained'' sheets from
\cite{MR3102597,MR3168627}. The modification consists in  considering clamped
boundary conditions (for displacements and deformation gradients), and dropping
the constraint on the deformation at the center of the sheet. This completely
changes the character of the problem, and the method of proof from
\cite{MR3102597,MR3168627} breaks down. 

\medskip

This is the variational problem  we will consider here, and we will prove an
energy scaling law for it, see Theorem \ref{thm:main} below. There is one caveat: We  penalize
the bending  energy in $L^p$ with $p\in (2,\frac83)$ (see \eqref{eq:13}), instead of, as
would be dictated by a heuristic derivation of the von K\'arm\'an model from three-dimensional elasticity, $p=2$. For a discussion of this modification, see Remark \ref{rem:main}.

\medskip

Our method of proof builds on the observations we have made in
\cite{HOregular,2015arXiv150907378O}, where we proved scaling laws for an
elastic sheet with a single disclination. The guiding principle is that the
(linearized) Gauss curvature is controlled by both the membrane and the bending
energy, in different function spaces. The boundary conditions can be used to
show that the Gauss curvature is bounded from below in a certain  space ``in
between'' in the sense of interpolation. In the
recent paper \cite{olber2017coneconv}, we show that for the setting of
\cite{HOregular,2015arXiv150907378O}, it is not necessary to use interpolation,
and lower bounds for the bending energy can be obtained by using the control
over the membrane energy alone. The present setting with a flat reference
metric however defines an interpolation type problem for the Gauss curvature,
and we hope that this approach can also yield results for similar variational problems.

\medskip

This paper is structured as follows: In 
Section \ref{sec:sett-stat-main}, we state our main result, Theorem
\ref{thm:main}. In Section \ref{sec:preliminaries}, we collect some facts
from
the literature, concerning the Brouwer degree, Sobolev and Triebel-Lizorkin
spaces, and interpolation theory. The proof of Theorem \ref{thm:main} is
contained in Section \ref{sec:proof-theor-refthm:m}.

\medskip

{\bf Notation.} We write $B_1=\{x\in\R^2:|x|<1\}$ and $S^1=\partial B_1$. 
When dealing with functions on $S^1$, we will  identify $S^1$ with the one
dimensional torus $\R/(2\pi\Z)$. 
For
$x=(x_1,x_2)\in \R^2$, we write $\hat x=x/|x|$ and $x^\bot=(-x_2,x_1)$. In
Section \ref{sec:sett-stat-main} below, we introduce a function $\beta\in
W^{2,p}(S^1)$ that can be considered as fixed for the rest of the paper. The
symbol ``$C$'' is used as follows: A statement such as ``$f\leq Cg$'' is shorthand for
``there exists a constant $C>0$ that only depends on $\beta$ such that $f\leq
Cg$''. The value of $C$ may change within the same line. 
For $f\leq Cg$, we also write $f\lesssim g$. The symmetrized gradient of a
function $u:U\to\R^2$ with $U\subseteq \R^2$ is denoted
by $\sym Du=\frac12 (Du+Du^T)$.

\section{Setting and statement of main theorem}
\label{sec:sett-stat-main}
Let $ \beta\in W^{2,p}(S^1)$ with 
\begin{equation}
\int_{S^1}\left(\beta^2(t)-\beta'^2(t)\right)\d t=0 \quad \text{ and
}\int_{S^1}|\beta+\beta''|\d t>0\,.\label{eq:4}
\end{equation}
Using the identification of $S^1=\{x\in\R^2:|x|=1\}$ with the torus
$\R/(2\pi\Z)$,  we define 
\begin{equation}
\begin{split}
  \gamma(t)&:=-\frac{\beta^2(t)}{2} \\
\zeta(t)&:=\frac12\int_0^t
  \beta^2(s)-\beta'^2(s)\d s\,\,,
\end{split}\label{eq:15}
\end{equation}
and we define
 $u_\beta:\R^2\to\R^2$ by
\begin{equation}
\begin{split}
  u_\beta(x)\cdot \hat x&:= |x|\gamma(\hat x)\,,\\
  u_\beta(x)\cdot \hat x^\bot&:= |x|\zeta(\hat x)\,,
\end{split}\label{eq:16}
\end{equation}
Furthermore, we set
\begin{equation}
v_\beta(x)=|x|\beta(\hat x)\,.\label{eq:18}
\end{equation}
Note that the deformation defined by $u_\beta,v_\beta$ is an isometric immersion
in the von K\'arm\'an sense, i.e., 
\[
\sym Du_\beta+\frac12 Dv_\beta\otimes Dv_\beta=0\,,
\]
but $D^2v_\beta\not\in L^p$ for $p\geq 2$.
The set of allowed configurations is given by
\[
\mathcal A_{\beta,p}:=\left \{(u,v)\in W^{1,2}(B_1)\times W^{2,p}(B_1): v=v_\beta,\,
Dv=Dv_\beta\text{ and } u=u_\beta \text{ on }S^1\right\}\,.\]
The energy is given by a sum of membrane and bending energy,
\begin{equation}
\Ihp(u,v)= \left\|\sym Du+\frac12 Dv\otimes Dv\right\|_{L^2(B_1)}^2+h^2 \|D^2v\|_{L^p(B_1)}^2\,.\label{eq:13}
\end{equation}
In the statement of our main theorem, the dual exponent $p'$ is defined as usual
by $\frac1p+\frac1{p'}=1$.
We are going to prove
\begin{theorem}
\label{thm:main}
Let $p\in (2,8/3)$. Then there exists a constant $C=C(\beta,p)>0$ such that
\[
C^{-1} h^{p'}\leq\inf_{y\in \mathcal A_{\beta,p}}\Ihp(y)\leq C h^{p'}\,.
\]
\end{theorem}

\begin{remark}
\label{rem:main}
\begin{itemize}
\item[(i)] The arguments of the energy functional $(u,v):B_1\to \R^3$ can be thought of as the
  displacements of a deformation $x\mapsto x+\e^2 u(x)+\e v(x)e_3$, where $\e$
  is another small parameter (with $h\ll \e$). The membrane energy
  geometrically corresponds to the deviation of the induced metric tensor from the flat
  Euclidean metric: The induced metric is given by $(\id_{2\times
    2}+\e^2Du+\e e_3\otimes Dv)^T(\id_{2\times 2}+\e^2 Du+\e e_3\otimes Dv )$, and the membrane
  term $\sym Du+\frac12 Dv\otimes Dv$ is the leading order term of the
  difference to the flat reference metric.
We say that $\det D^2 v$ is the ``linearized Gauss curvature'' since  we have that the Gauss curvature is
given by $K=\e^2\det D^2 v+o(\e^2)$.
Rigorously, the von K\'arm\'an energy (\eqref{eq:13} with $p=2$) has only been justified as a
limit of three-dimensional finite elasticity for small deformations
\cite{MR569597,MR2210909}. Nevertheless, it has a  long and successful history of
describing phenomena including moderate deformations. 
\item[(ii)]The conditions on
  the boundary values  in \eqref{eq:4} are the von K\'arm\'an version of the
  requirement that the associated conical deformation defined by
  $u_\beta,v_\beta$ has no membrane energy and
  is not contained in a plane.
\item[(iii)] The restriction to
  the range $p\in (2,\frac83)$ is due to our method of proof, which is an
  application of the Gagliardo-Nirenberg inequality  to the linearized Gauss
  curvature $\det D^2v$. This interpolation inequality is only
  valid for that range. The standard
  von K\'arm\'an model is linear in the material response, and hence it   
   penalizes  the bending energy in $L^2$. In this case, one expects an
  energy scaling law of the form $I_{h,2}\sim h^2\log \frac 1h$, as is the case
  when the center of the sheet is fixed (see \cite{MR3102597,MR3168627}). In
  order to obtain lower bounds for this case, one would  have to show
  ``additional regularity'', in the sense that one would need to control higher
  $L^p$ norms of $D^2v$ by the $L^2$ norm. One might hope that such estimates
  are possible e.g.~for minimizers of the functional. However, we do not know if this
  is possible.
\item[(iv)]  We do not know if our method of proof can be adapted to prove an analogous
  result for the geometrically fully nonlinear plate model that is given by the
  energy $\tilde I_{h,p}:W^{2,p}(B_1;\R^3)\to\R$, 
\[
\tilde I_{h,p}(y)=\int_{B_1}| Dy^TDy-\id_{2\times 2}|^2\d
x+h^2\|D^2y\|^2_{L^p(B_1)}\,.
\]
The reason is that it seems much more complicated to obtain a good test
function in $W^{1,p}$ for  $\sum_i\det D^2y_i$ (which is the appropriate
linearization of Gauss curvature in that setting) that would yield a lower bound
for this quantity in the Sobolev space $W^{-1,p'}$. In the von K\'arm\'an case, we
can simply use the identity 
\[(\div\psi)|_{Dv(x)}\det D^2 v(x)=\div \left(\psi(Dv(x))\cof
  D^2v(x)\right)
\]
and compute a lower bound for this quantity by Gauss' Theorem, using the
boundary values of $Dv$. In the case of $y\in W^{2,p}(B_1;\R^3)$, we cannot
argue similarly component by component: only the
\emph{sum} $\sum_i\det D^2y_i$ is controlled by the energy. The task is 
to find a test function that a)  allows us to use Gauss' Theorem and the
boundary values to obtain a lower bound of order one, and b)
is controlled in $W^{1,p}$ by the bending
energy. We have not found a way to do so.
\end{itemize}
\end{remark}

\section{Preliminaries}
\label{sec:preliminaries}
\subsection{The Brouwer degree}
At the heart of our proof of the lower bound for the energy is an interpolation
estimate for the linearized Gauss curvature. This quantity can be thought of as
a pull back of the volume form on $\R^2$ under the map $Dv: B_1\to \R^2$. This
is where the Brouwer degree becomes relevant, since integrals over the
linearized Gauss
curvature ``downstairs'' (on $B_1$) can be expressed as integrals ``upstairs''
(on $\R^2$) over the Brouwer degree of $Dv$.

For a bounded set $U\subset\R^n$, $f\in C^\infty(\overline U;\R^n)$ and
$y\in\R^n\setminus f(\partial U)$, the Brouwer degree $\deg(f,U,y)$ may be
defined as follows: Let $A_{y,f}$ denote the connected component of
$\R^n\setminus f(\partial U)$ that contains $y$, and let $\mu$ be a smooth
$n$-form on $\R^n$ with support in $A_{y,f}$ such that
$\int_{\R^n}\mu=1$. Then we set
\[
\deg(f,U,y)=\int_U f^\# \mu\,,
\]
where $f^\#$ denotes the pull-back under $f$. By approximation with smooth
functions, $\deg(f,U,y)$ may be defined for every $f\in C^0(\overline U;\R^n)$
and $y\in \R^n\setminus f(\partial U)$. 
If $f\in W^{1,\infty}(\overline U;\R^n)$ and 
$\mu$ is a $n$-form with of regularity $W^{1,\infty}$, 
it follows straightforwardly from the definition that
\[
\int_{\R^n}\deg(f,U,\cdot)\mu=\int_U f^\# \mu\,.
\]
If $\mu=\varphi\d z$, where $\d z$ is the canonical volume form on $\R^n$, this reads
\begin{equation} 
\int_{\R^n}\varphi(z)\deg(f,U,z)\d z=\int_U \varphi(f(x))\det Df(x)\d x\,.\label{eq:6}
\end{equation}
If $f\in C^1(U;\R^n)$, $U$ has Lipschitz boundary and $\mu$ is a smooth $n-1$-form on $\R^n$, then we have
\begin{equation}
\begin{split}
  \int_{\R^n}\deg(f,U,\cdot)\d\mu&= \int_{U}f^\#(\d\mu)\\
  &=\int_{\partial U} f^\# \mu\,.
\end{split}\label{eq:17}
\end{equation}
It can be shown that $y\mapsto
\deg(f,U,y)$ is constant on the connected components of $\R^n\setminus
f(\partial U)$.
Finally, we are going to use the fact that $\deg(f,U,y)$ only depends on
$f|_{\partial U}$. Thus for every continuous function $\tilde f:\partial U\to\R^N$,
and $y\not\in \tilde f(\partial U)$, we may define
\[
\degd(\tilde f,\partial U,y)=\deg(f,U,y)\,,
\]
where $f$ is any continuous extension of $\tilde f$ to $\overline U$.
For more details (in particular for the proofs of the statements made here) see \cite{MR1373430}.

\subsection{Function spaces}

Our main estimate for the Gauss curvature is a version of the
Gagliardo-Nirenberg inequality for the spaces $W^{-m,p}$ with $m\in\N$ and $p\in
(1,\infty)$.
To define these spaces, let $\Omega\subset\R^n$ be a bounded open set. For $u\in
L^1(\Omega)$ with compact support in $\Omega$, we set
\[
\|u\|_{W^{m,p}_0(\Omega)}:=\left(\int_\Omega |D^m u|^p\d x\right)^{1/p}\,.
\]
This defines a norm on the space $W^{m,p}_0(\Omega)$ which is defined as the set
of those $u\in L^1(\Omega)$ that are
compactly supported in $\Omega$ and satisfy
$\|u\|_{W^{m,p}_0(\Omega)}<\infty$. The dual space of $W^{m,p}_0(\Omega)$ is
denoted by $W^{-m,p'}(\Omega)$, where $p'$ satisfies
$\frac1p+\frac{1}{p'}=1$. The norm on $W^{-m,p'}(\Omega)$ is given by
\[
\|f\|_{W^{-m,p'}(\Omega)}=\sup \left\{\left<f,\varphi\right>:\varphi\in
W^{m,p}_0(\Omega),\, \|\varphi\|_{W^{m,p}_0(\Omega)}\leq 1\right\}\,.
\]
Additionally, we define the space $W^{m,p}(\R^n)$ as the completion of
$C_c^\infty(\R^n)$ under the norm
\[
\|u\|_{W^{m,p}(\R^n)}=\left(\int_{\R^n} |D^m u|^p\d x\right)^{1/p}\,.
\]

The Gagliardo-Nirenberg inequality that we want to prove is an  interpolation
inequality for the spaces $W^{-m,p'}(\Omega)$. In fact, the interpolation can be carried out in
the spaces $W^{m,p}$ (with $m\geq 0$ and the understanding $W^{0,p}\equiv L^p$). These will be derived
by taking recourse to results from the literature, where one finds a well developed interpolation theory for the
Triebel-Lizorkin spaces $F^s_{p,q}$, which contains the appropriate
interpolation between Lebesgue and Sobolev spaces
as a special case.

Let $\mathcal D'(\R^n)$ denote the space of temperate distributions on $\R^n$, and let
$\mathcal F:\mathcal D'(\R^n)\to \mathcal D'(\R^n)$ denote the Fourier
transform. We  briefly recall the Littlewood-Paley decomposition of temperate
distributions: Let $\eta_0\in C^\infty_c(\R^n)$ such that $0\leq \eta_0\leq 1$,
$\eta_0(x)=1$ for $|x|\leq 1$, $\eta_0(x)=0$ for $|x|\geq 2$. Set
$\eta_j(x)=\eta_0(2^{-j}x)-\eta_0(2^{-j+1}x)$ for $j\geq 1$.

\begin{definition}[\cite{triebel2006theory}, Chapter 2.3.1]
  For $-\infty<s<\infty$, $0<p,q<\infty$, let
\[
F^s_{p,q}(\R^n)=\left\{f\in \mathcal D'(\R^n):
  \|f\|_{F^s_{p,q}(\R^n)}:=\Big\|\left\|2^{sj}\mathcal F^{-1}\eta_j\mathcal F
      f\right\|_{l^q}\Big\|_{L^p(\R^n)}<\infty\right\}\,.
\]
\end{definition}

The following special cases of the Triebel-Lizorkin spaces will be relevant for
us (see \cite{triebel2006theory}, Sections 2.2.2 and 2.3.5):
\begin{equation}
  \label{eq:7}
  \begin{split}
    L^p(\R^n)&=F^{0}_{p,2}(\R^n)\\
W^{k,p}(\R^n)&=F^k_{p,2}(\R^n)\quad \text{ for }k\in\N\,.
    % W^{s,p}(\R^n)&=F^s_{p,p}(\R^n)\quad \text{ for }s\not\in\Z\,.
  \end{split}
\end{equation}

Apart from their interpolation properties, the following embedding theorem will
play a role in our proof:
\begin{theorem}[Theorem 2.7.1 in \cite{triebel2006theory}]
\label{thm:TLembedding}
Let $-\infty<s_1<s_0<\infty$, $0<p_0<p_1<\infty$ and $0<q_0,q_1<\infty$ such that
\[
s_1-\frac{n}{p_1}= s_0  -\frac{n}{p_0}\,.
\]
Then we have the continuous embedding
\[
F^{s_0}_{p_0,q_0}(\R^n)\subset F^{s_1}_{p_1,q_1}(\R^n)\,.
\]
\end{theorem}

\subsection{Real interpolation}
We recall some basic facts concerning the real interpolation method. Let
$X_0,X_1$ be Banach spaces such that there exists a topological vector space $Z$ with
continuous embeddings $X_0,X_1\subset Z$. In such a situation, let $t>0$ and
$x\in X_0+X_1$. We define
\[
K(t,x):=\inf\left\{ \|x_0\|_{X_0}+t\|x_1\|_{X_1}:\,x_0\in X_0,\,x_1 \in X_1,\, x_0+x_1=x\right\}\,.
\]
Let $0\leq \theta\leq 1$ and $p\geq 1$. The real interpolation space
$(X_0,X_1)_{\theta,p}$ is defined as
\[
(X_0,X_1)_{\theta,p}=\left\{x\in X_0+X_1: \Phi_{\theta,p}(x)<\infty\right\}\,,
\]
where
\[
\Phi_{\theta,p}(x)=\begin{cases}\left(\int_0^\infty |t^{-\theta}K(t,x)|^p\frac{\d
    t}{t}\right)^{1/p}& \text{ if }p<\infty\\
\sup_{t>0}|t^{-\theta}K(t,x)|&\text{ else}\,.\end{cases}
\]
The interpolation space $(X_0,X_1)_{\theta,p}$ is a normed space with the norm
$\Phi_{\theta,p}(x)$. For every $p<\infty$, we have  the continuous embedding
\begin{equation}
(X_0,X_1)_{\theta,p}\subset (X_0,X_1)_{\theta,\infty}\,.\label{eq:8}
\end{equation}
For a proof, see e.g.~Chapter 1.3 of  \cite{triebel1978interpolation}.
Concerning real interpolation of Triebel-Lizorkin spaces, we have the following theorem: 

\begin{theorem}[\cite{triebel1978interpolation}, Theorem 1 in Chapter 2.4.2]
\label{thm:Triebelinter}
Let $-\infty<s_0,s_1<\infty$, $1<p_0,p_1,q_0,q_1<\infty$, $0<\theta<1$ and
\[
s=(1-\theta) s_0+\theta s_1,\quad
\frac{1}{p}=\frac{1-\theta}{p_0}+\frac{\theta}{p_1}\,.
\]
Then we have
\[
\left(F^{s_0}_{p_0,q_0}(\R^n),F^{s_1}_{p_1,q_1}(\R^n)\right)_{\theta,p}=F^s_{p,p}(\R^n)\,.
\]
\end{theorem}

\section{Proof of Theorem \ref{thm:main}}
\label{sec:proof-theor-refthm:m}
A sketch of the proof of Theorem \ref{thm:main} goes as follows: As usual, the upper bound
is provided by a conical construction  that is smoothed on a ball around the
origin with  the appropriate
length scale, see Lemma \ref{lem:upperbd}. At the heart of the lower bound, we have an interpolation
inequality for the linearized Gauss curvature $\det D^2v$. Formally, the
Gagliardo-Nirenberg inequality \cite{MR0109940} yields
\begin{equation}
  \begin{split}
    \|\det D^2v\|_{W^{-1,p'}(B_1)}&\lesssim
    \|\det D^2v\|_{W^{-2,2}(B_1)}^{1-\alpha}\|\det D^2v\|_{L^{p/2}(B_1)}^\alpha% \\
    % &\quad+\|\det D^2 v\|_{W^{-2,2}(B_1)}
  \end{split}
\quad \text{ (formally) }\,,\label{eq:1}
\end{equation}
with $\alpha\in [\frac12,1]$ determined by
\[
\frac2{p'}-1=\left(\frac{2}{p/2}-2\right)\alpha+1-\alpha\,\,,
\]
i.e., 
\[
\alpha=\frac{2}{3p-4}\,.
\]
In equation \eqref{eq:1}, the left hand side can be bounded from below using the
boundary conditions and an argument involving the mapping degree. Namely, for an
appropriately chosen test function $\varphi\in C_c^\infty(\R^2)$, we have
\[
\int_{B_1}\varphi\circ Dv(x)\det D^2 v\d x=\int_{\R^2}\varphi(z)\deg(Dv,B_1,z)\d
z=O(1)\,.
\]
For the details 
see Lemma
\ref{lem:bdrylem}. 

The exponents in \eqref{eq:1} are chosen such that the terms on the right hand side  can be estimated by the energy,
\begin{equation}
\begin{split}
  \|\det D^2v\|_{W^{-2,2}}&\lesssim \left\|\sym Du+\frac12 Dv\otimes Dv\right\|_{L^2}
  \lesssim \Ihp(u,v)^{1/2}\,,\\
  \|\det D^2v\|_{L^{p/2}}&\lesssim \|D^2 v\|_{L^p}^2\lesssim h^{-2}\Ihp(u,v)\,.
\end{split}\label{eq:2}
\end{equation}

With these estimates, we obtain the desired lower bound.

Basically, all that remains is to prove the aforementioned lemmas, and justify
\eqref{eq:1}. We could not find a
proof of the Gagliardo-Nirenberg inequality for ``negative
orders of differentiation''  in the literature. We believe that it holds true,
and that a proof could be given using the machinery from
\cite{triebel1978interpolation}. However, in our case  a shorter route
exists, using  the fact that $v:B_1\to\R$ has a natural extension to $\R^2$
with vanishing membrane energy on $\R^2\setminus B_1$, and  existing
results on  interpolation of Sobolev and Triebel-Lizorkin spaces on $\R^n$ (see
again  \cite{triebel1978interpolation}).

% From equation \eqref{eq:1}, one may now derive a lower bound for the energy as follows,
% \[
% \begin{split}
%    O(1)=\int_{\R^2}\varphi(z)\deg(Dv,B_1,z)\d z&=\int_{B_1}\det D^2v(x) \varphi(Dv(x))\d x\\
%     &\leq \|f\|_{W^{-1,p'}(B_1)} \|\varphi_\beta\|_{W^{1,\infty}}\|D^2v\|_{L^p}\,.\\
%   &\leq \|f\|_{W^{-1,p'}(B_1)} \|\varphi_\beta\|_{W^{1,\infty}}h^{-1}\Ihp(u,v)^{1/2}\,.
%   \end{split}
% \]
% Inserting  \eqref{eq:1} and \eqref{eq:2} in this last estimate, we obtain
% \[
% C\leq h^{-1}\Ihp^{1/2}\Ihp^{(1-\alpha)/2}h^{-2\alpha} \Ihp^\alpha\,,
% \]  
% which implies 
% \[
% \Ihp(u,v)\gtrsim h^{(1+2\alpha)/(1+\alpha/2)}=h^{p/(p-1)}=h^{p'}\,.
% \]
 
\bigskip

Now we start with the proof.

\begin{lemma}
\label{lem:upperbd}
 We have
\[
\inf_{y\in \mathcal A_{\beta,p}}\Ihp(y)<C h^{p'}\,.
\]
\end{lemma}
\begin{proof}
Recall the definition of   $u_\beta,v_\beta$ from \eqref{eq:15}-\eqref{eq:18}.
% Using the identification of $S^1=\{x\in\R^2:|x|=1\}$ with the torus
% $\R/(2\pi\Z)$,  we define 
% \begin{equation}
% \begin{split}
%   \gamma(t)&:=-\frac{\beta^2(t)}{2} \\
% \delta(t)&:=\frac12\int_0^t
%   \beta^2(s)-\beta'^2(s)\d s\,\,,
% \end{split}\label{eq:15}
% \end{equation}
% and
% $u_\beta:B_1\to\R^2$ by
% \begin{equation}
% \begin{split}
%   u_\beta(x)\cdot \hat x&:= |x|\gamma(\hat x)\,,\\
%   u_\beta(x)\cdot \hat x^\bot&:= |x|\delta(\hat x)\,,
% \end{split}\label{eq:16}
% \end{equation}
Let $\eta\in C^\infty([0,\infty))$ with $\eta(t)=0$ for $t<\frac12$, $\eta(t)=1$
for $t\geq 1$. We set 
\[
v_{\beta,h}(x)=\eta\left(\frac{|x|}{h^{p'/2}}\right)v_\beta(x)\,.
\]
Now we have
\[
\left|\sym Du_\beta(x)+\frac12 Dv_{\beta,h}(x)\otimes
  Dv_{\beta,h}(x)\right|=\begin{cases}0 &\text { if } |x|\geq h^{p'/2}\\
O(1) &\text{ else.}\end{cases}
\]
Furthermore, we have
\[
\begin{split}
  \int_{B_1} |D^2v_{\beta,h}|^p&= \int_{B_1\setminus B_{h^{p'/2}}} \d
  x\left|\frac{\beta''(\hat x)+\beta(\hat
      x)}{|x|}\right|^p +\int_{ B_{h^{p'/2}}} O(h^{-p(p'/2)})\d x\\
    &\lesssim h^{(2-p)p'/2}\,.
  \end{split}
\]
This implies
\[
\begin{split}
  \Ihp(u_\beta,v_{\beta,h})&=\int_{B_1} \left|\sym
    Du_\beta(x)+\frac12 Dv_{\beta,h}(x)\otimes
    Dv_{\beta,h}(x)\right|^2\d x\\
&\quad +h^2\left(\int_{B_1} |D^2v_{\beta,h}|^p\d x\right)^{2/p}\\
  &\lesssim h^{p'}\,.
\end{split}
\]
This proves the lemma.
\end{proof}

% In the following lemma, we are going to use some standard notation from
% geometric measure theory as in \cite{MR0257325}. The current that is defined by
% integration over $\R^2$ with the two-dimensional Lebesgue measure

\begin{lemma}
\label{lem:bdrylem}
Assume that $\beta\in W^{2,p}(S^1)$ with
\[
\int\beta^2(t)-\beta'^2(t)\d t=0\quad\text{ and }\quad\int|\beta+\beta''|\d
t\neq 0\,,
\]
and let $v_\beta$ be defined by \eqref{eq:18}. Then
there exists $\varphi_\beta\in C_c^\infty(\R^2)$ such
that $\supp \varphi_\beta\cap Dv_\beta(S^1)=\emptyset$ and
\[
\int_{\R^2}\varphi_\beta(z)\degd(Dv_\beta,S^1,z)\d z>0\,.
\]
% \item[(ii)] For every $v\in \mathcal A_{\beta,p}$, we have
% \[
% \deg(Dv,B_1,\cdot)=\deg(Dv_\beta,B_1,\cdot)\,.
% \]
%   \end{itemize}
\end{lemma}

\begin{proof}

\emph{Step 1:} Reduction to the smooth case. We claim that we may assume that $\beta\in C^\infty(S^1)$. Indeed, for every $\e>0$ we may choose $\tilde\beta\in C^\infty(S^1;\R^2)$ such that
\[
\|\beta-\tilde\beta\|_{W^{2,p}}<\e \quad\text{ and }\quad
\int|\tilde\beta+\tilde \beta''|\d t\neq 0.
\]
Additionally, we may choose $\tilde \beta$ such that
\[
\int_{S^1}(\tilde \beta^2-\tilde \beta'^2)\d t= 0 \,.
\]
We have
\[
Dv_\beta= \beta(\hat x)\hat x+\beta'(\hat x)\hat x^\bot\,,
 \quad 
Dv_{\tilde \beta}= {\tilde \beta}(\hat x)\hat x+{\tilde \beta}'(\hat x)\hat x^\bot\,.
\]
By the continuous embedding $W^{2,p}\to C^1$, we have that
$\|Dv_\beta-Dv_{\tilde\beta}\|_{C^0(S^1)}$ 
and hence also $\|\degd(Dv_\beta,S^1,\cdot)-\degd(Dv_{\tilde \beta},S^1,\cdot)\|_{L^1(\R^2)}$
can be made arbitrarily small by a suitable choice of $\e$. If we manage to show $\degd(Dv_{\tilde
  \beta},S^1,\cdot)\neq 0$ in $L^1(\R^2)$, then we have also proved the claim of
the lemma. Hence, from now on we prove the claim of the lemma for $\beta\in C^\infty(S^1)$.

\medskip

\emph{Step 2:} Taking the derivative of ``$\deg$''. 
For $t\in S^1$, let $e_t=(\cos t,\sin t)$.
Let $\gamma:S^1\to\R^2$ be defined by
\[
\gamma(t)= \beta(t) e_t+\beta'(t) e_t^\bot\,.
\]
% Let $\Gamma$ be any $C^\infty$ extension of $\gamma$ to $\overline B_1$.
It is enough to show that  $\degd(Dv_\beta,S^1,\cdot)=\degd(\gamma,S^1,\cdot)$ is
non-zero in $L^1(\R^2)$. By \eqref{eq:17}, we have for any smooth one form
$\omega=\omega_1\d x_1+\omega_2\d x_2$ on $\R^2$:
\[
\begin{split}
  \int_{\R^2}\degd(\gamma,S^1,\cdot)\d \omega
  % \int_{B_1} \Gamma^\# \d \omega\\
  &=\int_{S^1} \gamma^\#\omega\,.
\end{split}
\]
% We have
% \[
% \E^2\ecke\deg(\Gamma,B_1,\cdot)=\Gamma^\#(\E^2\ecke B_1)\,,
% \]
% and the boundary of that current is given by
% \[
% \begin{split}
%   \partial\left(\E^2\ecke\deg(\Gamma,B_1,\cdot)\right)
%   &=\Gamma^\#(\partial(\E^2\ecke B_1))\\
%   &=\gamma^\#(\H^1\ecke S^1)\,.
% \end{split}
% \]
If we show that the right hand side is  non-zero for some choice of $\omega$,  we are done.
% For a one-form $\eta=\eta_1\d x_1+\eta_2 \d x_2$ on $\R^2$, we have

% \begin{equation}
% \begin{split}
%   \gamma^\#(\H^1\ecke S^1)(\eta)&=
%   (\H^1\ecke S^1)(\gamma_\#\eta)\\
%   &=\int_0^{2\pi} \eta_1(\gamma(t))\gamma_1'(t)+\eta_2(\gamma(t))\gamma_2'(t)\d
%   t\\
% &= \int_0^{2\pi} \left(\eta_1\left(\gamma(t)\right)\gamma_1'(t)+\eta_2(\gamma(t))\gamma_2'(t)\right)\d
%   t\,.
% \end{split}\label{eq:5}
% \end{equation}
Let $f:\R^2\to \R^2$ be defined by
\begin{equation}
\begin{split}
  x&\mapsto \sum_{t\in\gamma^{-1}(x)}\gamma'(t)\\
  &= \sum_{t\in\gamma^{-1}(x)}(\beta(t)+\beta''(t))e_t^\bot\,.
\end{split}\label{eq:9}
\end{equation}
Then we % see from \eqref{eq:5} that
we have
\[
(\gamma^\#\omega)(t) = (\omega_1(\gamma(t)),\omega_2(\gamma(t)))\cdot f(\gamma(t)) \d t\,,
\]
and we see that it suffices to show that $f\neq 0$ on a set of positive $\H^1$
measure to prove the claim of the lemma.

\medskip

\emph{Step 3:} Proof of the lemma by contradiction.
We assume that $f=0$ $\H^1$-almost everywhere and show that this leads to a contradiction.
Since $\gamma'(t)=(\beta(t)+\beta''(t))e_t^\bot$, we have that $\gamma'(t)=0$ if
and only if $\beta(t)+\beta''(t)=0$.
Let $U$ be an open interval
such that $\gamma'\neq 0$ on $U$ and that
$\gamma:U\to\gamma(U)$ is a diffeomorphism. 
Our aim is now to show that up to $\H^1$ null sets, we have
\[
\gamma^{-1}(\gamma(U))\setminus U=U+\pi\,,
\]
where we are using the identification of $S^1$ with $\R/(2\pi\Z)$.

\medskip

By $f=0$ $\H^1$-almost everywhere  on
$\gamma(U)$ and the explicit form \eqref{eq:9} of $f$, there exists $E_1\subset
S^1$ with $\H^1(E_1)=0$ such that

\begin{equation}
\gamma(U\setminus E_1)\subseteq \gamma(S^1\setminus
U)\label{eq:10}\,.
\end{equation}
Next let 
\[
\begin{split}
  E_2&:=\{t\in S^1:\gamma'(t)=0\}\,,\\
  A&:=\gamma(E_2)\,.
\end{split}
\]
By Sard's Lemma, we have $\H^1(A)=0$. Furthermore, let
\[
E_3:=\gamma^{-1}(\gamma(U\setminus E_1)\setminus A)\setminus U\,,
\]
and let $E_4\subseteq E_3$ be the set of points that are not of density one,
i.e.,  
\[
E_4:=\left\{x\in E_3:\,\liminf_{\e\to 0} \frac{\H^1(( x-\e,x+\e)\cap E_3)}{2\e}<1\right\}\,.
\]
It is a well know fact from measure theory that $\H^1(E_4)=0$. Let
$E_5:=\gamma^{-1}(\gamma(E_4))\cap U$. Then also $\H^1(E_5)=0$.

\medskip

Now let $p\in U\setminus (E_1\cup E_5)$. Then $\gamma(p)\not\in A\cup \gamma(E_4)$, and
by \eqref{eq:10}, we have
\[
\gamma(p)\in \gamma(S^1\setminus U)\setminus (A\cup\gamma(E_5))\,.
\]
Hence there exists
$p'\in E_3\setminus E_4$ with $\gamma(p')=\gamma(p)$.
We may choose a sequence $p_k'$, $k\in\N$ with $p_k'\in E_3$,
$\gamma(p_k')\neq\gamma(p')$, and $p_k'\to p'$. Since $\gamma|_U$ is a
diffeomorphism, we may set
\[
p_k:=\gamma|_U^{-1}(\gamma(p_k'))
\]
and obtain a sequence $p_k\to p$ in $U$, with $\gamma(p_k)\neq \gamma(p)$.
Now we have for every $k$:
\[
\frac{\gamma(p_k)-\gamma(p)}{|\gamma(p_k)-\gamma(p)|}
=\frac{\gamma(p_k')-\gamma(p')}{|\gamma(p_k')-\gamma(p')|}\,.
\]
Passing to a suitable subsequence and taking the limit $k\to\infty$ in that equation, we obtain that the vectors
$\gamma'(p)$ and $\gamma'(p')$ are parallel. Since
\[
\gamma'(t)=(\beta(t)+\beta''(t))e_t^\bot\,, 
\]
and $p\neq p'$, we must have $e_p=-e_{p'}$, and hence (using the identification of
$S^1$ with $\R/(2\pi\Z)$)
\[
p'=p+\pi\,.
\]
Summarizing, we have shown that 
for $\H^1$-almost every $p\in U$, $p+\pi\in \gamma^{-1}(\gamma(U))\setminus
U$. We also may conclude that for every $p'\in E_3\setminus E_4$, $p'+\pi\in
U$. Hence, as desired, we have
\[
\gamma^{-1}(\gamma(U))\setminus U=U+\pi\quad\text{up to $\H^1$ null sets.}
\]
Since for every $x\in S^1\setminus E_2$ there exists a neighborhood $U$ of $x$ with the
properties we have assumed above, we obtain that for $\H^1$-almost every
 $t\in S^1\setminus E_2$,  we have $\gamma(t)=\gamma(t+\pi)$.
Hence,
\[
\beta(t)e_t+\beta'(t)e_t^\bot=-\beta(t+\pi)e_t-\beta'(t+\pi)e_t^\bot\,,
\]
which implies 
\begin{equation}
\beta(t+\pi)=-\beta(t),\quad\beta'(t+\pi)=-\beta'(t)\quad\text{ for $\H^1$-a.e. }t\in S^1\setminus E_2\,.\label{eq:12}
\end{equation}
We claim that we even have
\begin{equation}
\beta(t+\pi)=-\beta(t)\quad\text{ for  } t\in S^1\,.\label{eq:11}
\end{equation}
Indeed, let $t\in S^1$. If $t\in \overline {S^1\setminus E_2}$, then the claim
follows from \eqref{eq:12}. If $t$ is in the interior of $E_2$, then let
$T\in \partial E_2$ such that $(t,T)\subset E_2$. Then we have that also
$(t+\pi,T+\pi)\subset E_2$, and $\beta(T+\pi)=-\beta(T)$,
$\beta'(T+\pi)=-\beta'(T)$. The values of $\beta(t),\beta(t+\pi)$ are then
determined by the initial values of $\beta,\beta'$ at the points $T,T+\pi$ and by the
ODE $\beta+\beta''=0$. By the linearity and translation invariance of this initial value problem, we obtain
$\beta(t+\pi)=-\beta(t)$ as desired. This proves the claim \eqref{eq:11}.

By \eqref{eq:11}, we have $\int_{S^1}\beta(t)\d t=0$. By the
Poincar\'e-Wirtinger inequality, we have that
\[
\int_{S^1}(\beta^2-\beta'^2)\d t\leq 0\,,
\]
with equality only if $\beta$ is of the form $\beta(t)=C\sin (t+\alpha)$ for
some $C,\alpha\in\R$. Equality must hold true by assumption, which yields
\[
\beta+\beta''=0 \text{ on } S^1\,,
\]
in contradiction to our assumptions. This proves the lemma.
\end{proof}

\begin{lemma}
\label{lem:interlem}
Let $p\in (2,8/3)$, and 
\[
\theta=1-2/(3p-4)\,,\quad
\frac{1}{q}=\frac{1-\theta}{p/(p-2)}+\frac{\theta}{2}\,.
\]
Then we have
\[
W^{1,p}(\R^2)\subset \left(L^{p/(p-2)}(\R^2) ,W^{2,2}(\R^2)\right)_{\theta,q}\,.
\]
\end{lemma}
\begin{proof}
By \eqref{eq:7}, we have $L^{p/(p-2)}(\R^2)=F^0_{p/(p-2),2}(\R^2)$ and
$W^{2,2}(\R^2)=F^2_{2,2}(\R^2)$.
By Theorem \ref{thm:Triebelinter}, we obtain
\[
\left(L^{p/(p-2)}(\R^2) ,W^{2,2}(\R^2)\right)_{\theta,q}=F^{2\theta}_{q,q}(\R^2)\,.
\]
Finally, by Theorem \ref{thm:TLembedding}, we have
\[
W^{1,p}(\R^2)=F^1_{p,2}(\R^2)\subset F^{2\theta}_{q,q}(\R^2)\,.
\]
Note that the assumption $s_1<s_0$ in Theorem \ref{thm:TLembedding} is fulfilled
by $1>2\theta$, which in turn is a consequence of $p\in (2,8/3)$.
This proves the lemma.
\end{proof}

In the next  lemma, we use the following notation: For
$(u,v)\in\A_{\beta,p}$, we let $\bar v:\R^2\to \R$ be   defined by
\[
\bar  u(x)=\begin{cases} u(x)&\text{ if }x\in B_1\\
u_\beta( x) &\text{ if }x\in \R^2\setminus B_1\,,\end{cases}
\qquad
\tilde v(x)=\begin{cases} v(x)&\text{ if }x\in B_1\\
v_\beta( x) &\text{ if }x\in \R^2\setminus B_1\,,\end{cases}
\]
where  $u_\beta, v_\beta$ have been defined  in the 
equations \eqref{eq:15}-\eqref{eq:18}.
\begin{lemma}
\label{lem:vwh}
Let $(u,v)\in \A_{\beta,p}$. Then
\[
\|\det D^2 \bar v\|_{W^{-2,2}(\R^2)}\lesssim \left\|\sym Du+\frac12 Dv\otimes
Dv\right\|_{L^2(B_1)}\,.
\]  
\end{lemma}
\begin{proof}
  We write down the Hessian determinant of $\bar v$ in its very weak form,
\[
\det D^2 {\bar v}= ({\bar v}_{,1}{\bar v}_{,2})_{,12}-\frac12 ({\bar v}_{,1}^2)_{,22}-\frac12
({\bar v}_{,2}^2)_{,11}-=-\frac12\curl\curl D{\bar v}\otimes D{\bar v}\,.
\]
Here, we have introduced $\curl (w_1,w_2)=w_{1,2}-w_{2,1}$. (In the formula
above,  $\curl$ is first applied in each row of the matrix $D{\bar v}\otimes D{\bar v}$, and
then on the components of the resulting column vector.)
Since we have $\curl\curl (D w^T+D w)=0$ for every $w\in W^{1,2}(B_1;\R^2)$, we
obtain
\[
\det D^2 \bar v=-\curl\curl \left(\sym D\bar u +\frac12 D\bar v\otimes D\bar v\right)\,.
\]
We note that 
\[
\sym D\bar u +\frac12 D\bar v\otimes D\bar v= 0\quad\text{ on }\R^2\setminus
B_1\,.
\]
Hence for every $\varphi\in W^{2,2}(\R^2)$, we obtain by two integrations by
parts, and the Cauchy-Schwarz inequality,
\[
\begin{split}
  \int_{\R^2}\det D^2 \bar v\,\varphi\,\d x&=
  -\int_{\R^2} \left(\sym D\bar u +\frac12 D{\bar v}\otimes D{\bar v}\right):\cof D^2\varphi\d x\\
  &\leq \|\sym D u+\frac12 D{ v}\otimes D{ v}\|_{L^2(B_1)}
  \|\varphi\|_{W^{2,2}(\R^2)}\,.
\end{split}
\]
This proves our claim.
\end{proof}

% In order to be able to apply the interpolation results on $\R^2$ instead of
% $B_1$, we use the following lemma:
% \begin{lemma}
% \label{lem:extensiontrick}
%   Let $(u,v)\in \mathcal A_{\beta,p}$, and let $\psi\in W^{2,2}(\R^2)$. Then
% \[
% \left|\int_{B_1}\det D^2 v(x)\psi(x)\d x\right|\leq \|\det D^2
% v\|_{W^{-2,2}(B_1)}\|\psi\|_{W^{2,2}(\R^2)}\,.
% \]
% \end{lemma}
% \begin{proof}
%   In order to estimate $\int_{B_1}\det D^2v(x)\psi_1(x)\d x$
% without  any (possibly large) boundary terms, we introduce
%  $\tilde u:\R^2\to \R^2$, 
% By $(u,v)\in\mathcal A_{\beta,p}$,  we have $\tilde u\in
% W^{1,2}_{\mathrm{loc}}(\R^2;\R^2)$, $\tilde v\in
% W^{2,p}_{\mathrm{loc}}(\R^2)$, and  $\sym D\tilde u+\frac12 D\tilde
% v\otimes D\tilde v=0$ on $\R^2\setminus B_1$. Furthermore, 
% \[
% \det D^2\tilde v(x)=\begin{cases}\det D^2 v(x)& \text{ if } x\in B_1\\
% 0&\text{ else.}\end{cases}
% \]
% % We define 
% % as follows: For $x\in B_1$, $\tilde u(x)=u(x)$. On $\R^2\setminus B_1$, let
% % $\tilde u=u_\beta$,  This definition satisfies 
% Now we may compute
% \[
% \begin{split}
%   \int_{B_1}\det D^2 v(x)\psi(x)\d x
%   &=\int_{\R^2}\det D^2 \tilde v(x) \psi(x)\d x\\
%   &\leq \int_{\R^2}\left|\left(\sym D\tilde u+\frac12 D\tilde v\otimes D\tilde
%       v\right):\cof D^2\psi\right|\d x\\
%   &\lesssim \|\sym D\tilde u+\frac12 D\tilde v\otimes D\tilde
%   v\|_{L^2(B_1)}\|\psi\|_{W^{2,2}(\R^2)}\\
%   &\leq \|\det D^2 v\|_{W^{-2,2}(B_1)}\|\psi\|_{W^{2,2}(\R^2)}\,,
% \end{split}
% \]
% where we have used the integration by parts as in the proof of  the last lemma.
% This proves the claim.
% \end{proof}

\begin{proof}[Proof of Theorem \ref{thm:main}] The upper bound has been proved
  in Lemma \ref{lem:upperbd}. It remains to prove the lower bound.

For any $(u,v)\in\A_{\beta,p}$ we have
$Dv|_{S^1}=Dv_\beta|_{S^1}$, and hence
$\deg(Dv,B_1,\cdot)=\degd(Dv,S^1,\cdot)=\degd(Dv_\beta,S^1,\cdot)$. By Lemma
\ref{lem:bdrylem}, we may choose
$\varphi\in C^\infty_c(\R^2)$ such that $\varphi\circ Dv\in
W^{1,p}_0(B_1)\subset W^{1,p}(\R^2)$ and
\[
\begin{split}
  0&<C(\beta)\\
&=\int_{\R^2}\varphi(z)\deg(Dv,B_1,z)\d z\\
& =\int_{B_1}\det D^2v(x)
  \varphi(Dv(x))\d x\\
&=\int_{\R^2}\det D^2\bar v(x)
  \varphi(D\bar v(x))\d x\,,
\end{split}
\]
where we have used the notation introduced above Lemma \ref{lem:vwh}, and the
fact that $\det D^2 \bar v=0 $ on $\R^2\setminus B_1$.
% Additionally we may choose $\varphi$ such that $\varphi\circ Dv\in
% ; one of the  advantages of such a choice is  that from now on we
% can consider $\varphi\circ Dv$ as a function in $W^{1,p}(\R^2)$ with the
% respective norm.

% Now we use real interpolation of Triebel-Lizorkin spaces, see
% \cite{triebel2006theory} Chapter 2.4. 
% Let 
% \[
% \theta=1-\frac{2}{3p-4}\,,\quad \frac1q=\frac1p+\frac12-\frac{2}{3p-4}\,.
% \]
% With these choices, we have
% \[
% \frac{1}{q}= \frac{1-\theta}{p/(p-2)}+\frac{\theta}{2}\,.
% \]
% By Theorem 1 in Chapter 2.4.2 in \cite{triebel1978interpolation}, we have
% \begin{equation}
% \begin{split}
%   \left(L^{p/(p-2)}(\R^n),W^{2,2}(\R^n)\right)_{\theta,q}&=
%   \left(F_{p/(p-2),2}^0(\R^n),F^2_{2,2}(\R^n)\right)_{\theta,q}\\
%   &=F^{2\theta}_{q,q}(\R^n)\\
% &=W^{2\theta,q}(\R^n)\,.
% \end{split}\label{eq:3}
% \end{equation}
% For the identification of $W^{s,q}$ with $F^s_{q,q}$ for non-integer $s$, see 
% again \cite{triebel2006theory}, and also for the identification of $W^{k,p}$
% with $F^k_{p,2}$ for integer $k$. 
By Lemma \ref{lem:interlem}, $\psi:=\varphi\circ Dv\in
\left(L^{p/(p-2)}(\R^2),W^{2,2}(\R^2)\right)_{\theta,q}$. Hence by \eqref{eq:8},
there exist
functions $\psi_0: \R^+\to L^{p/(p-2)}(\R^2)$ and $\psi_1:\R^+\to
W^{2,2}(\R^2)$ such that $\psi_0(t)+\psi_1(t)=\psi$ for all $t\in\R^+$ and
\[
t^{-\theta}\|\psi_0(t)\|_{L^{p/(p-2)}(\R^2)}+t^{1-\theta}\|\psi_1(t)\|_{W^{2,2}(\R^2)}\lesssim \|\psi\|_{W^{1,p}(\R^2)}\,.
\] 
Rearranging, we have for every $t>0$ that
\[
\begin{split}
  \|\psi_0(t)\|_{L^{p/(p-2)}(\R^2)}&\lesssim t^\theta \|\psi\|_{W^{1,p}(\R^2)}\\
  \|\psi_1(t)\|_{W^{2,2}(\R^2)}&\lesssim t^{\theta-1} \|\psi\|_{W^{1,p}(\R^2)}\,.
\end{split}
\]
Now we fix the argument,
\[
t:=\frac{\|\det D^2 \bar v\|_{W^{-2,2}}}{\|\det D^2 \bar v\|_{L^{p/2}}}\,,
\]
and write
$\psi_0=\psi_0(t),\psi_1=\psi_1(t)$. % Note that we possibly have $\psi_1\not\in
% W^{2,2}_0(B_1)$. However, by Lemma \ref{lem:extensiontrick}, we have
% nevertheless
% \[
% \left|\int_{B_1}\det D^2 v(x)\psi_1(x)\d x\right|\leq \|\det D^2
% v\|_{W^{-2,2}(B_1)}\|\psi_1\|_{W^{2,2}(\R^2)}\,.
% \]
Hence we
 may estimate
\begin{equation}
  \begin{split}
    C(\beta)&=\int_{\R^2}\det D^2\bar v(x) \varphi(D\bar v(x))\d x\\
&\lesssim  \|\det D^2 \bar v\|_{L^{p/2}}\|\psi_0\|_{L^{p/(p-2)}}+\|\det D^2 \bar
v\|_{W^{-2,2}}\|\psi_1\|_{W^{2,2}}\\
  &\lesssim \|\det D^2 \bar v\|_{W^{-2,2}}^\theta\|\det D^2
  \bar v\|_{L^{p/2}}^{1-\theta}\|\psi\|_{W^{1,p}}\\
&\lesssim \Ihp^{\theta/2} (h^{-2}\Ihp)^{1-\theta}(h^{-2} \Ihp)^{1/2}\\
&\lesssim \Ihp^{(3-\theta)/2}h^{2\theta-3}\,,
\end{split}
\end{equation}
where we have used Lemma \ref{lem:vwh} and the facts
\[
|\det D^2v|\leq |D^2v|^2,\quad
\det D^2\bar v=0 \text{ on }\R^2\setminus B_1\]
to obtain the fourth line from the third.
This implies 
\[
\Ihp\gtrsim h^{(6-4\theta)/(3-\theta)}=h^{p/(p-1)}=h^{p'}\,,
\]
which proves the theorem.
\end{proof}

\section*{Acknowledgments}
This work has been supported by  Deutsche Forschungsgemeinschaft (DFG, German
Research Foundation) as part of project 350398276.

% \section*{Conflict of interest statement}
% The author declares that there is no conflict of interest.

\bibliographystyle{plain}
\bibliography{regular}

\end{document}